\documentclass[12pt]{article}
\usepackage{times}
\usepackage{amsfonts}
\usepackage{amssymb}
\usepackage{amsmath}
\usepackage{amsthm}
\usepackage[linesnumbered,lined,boxed]{algorithm2e}
\usepackage{color}
\usepackage{enumitem}
\usepackage{graphicx}
\usepackage{cite}
\usepackage{authblk}
\usepackage{mathtools}
\usepackage{url}

\newtheorem{theorem}{Theorem}
\newtheorem{lemma}[theorem]{Lemma}
\newtheorem{proposition}[theorem]{Proposition}

\newcommand{\real}{{\mathbb R}}

\newcommand{\inner}[2]{\left\langle#1,#2\right\rangle}
\newcommand{\Inner}[2]{\left\langle{#1},{#2}\right\rangle}
\newcommand{\biginner}[2]{\big\langle{#1},{#2}\big\rangle}

\DeclareMathOperator{\dom}{dom}
\DeclareMathOperator{\ri}{ri}

\DeclareMathOperator{\diag}{diag}

\DeclareMathOperator{\conv}{conv}

\DeclareMathOperator*{\argmin}{arg\,min}

\DeclareMathOperator{\suchthat}{ST}

\newcommand{\set}[2]{\left\{#1 \; \left|\;\; #2 \right.\right\}}

\newcommand{\norm}[1]{\left\|#1\right\|}

\newcommand{\card}[1]{\left|#1\right|}

\DeclareMathOperator{\proj}{proj}
\newcommand{\orthog}{^{\perp}}

\newcommand{\transpose}{^{\scriptscriptstyle\top}}

\DeclareMathOperator{\identity}{Id}

\mathchardef\mhyphen="2D

\makeatletter
\newcommand{\svdots}{%
  \vbox{\fontsize{\sf@size}{\sf@size pt}\linespread{0.3}\selectfont
    \kern0.2\baselineskip
    \hbox{.}\hbox{.}\hbox{.}%
    \kern0.1\baselineskip
  }%
}
\makeatother

\begin{document}

\title{Asynchronous Parallel Computation of Lagrangian Bounds for 
       Mixed-Integer Stochastic Programs
       by Frank-Wolfe Methods within Projective Hedging}
\author[1]{Jonathan Eckstein}
\author[2]{Jean-Paul Watson}
\author[3]{David L. Woodruff}
\affil[1]{\small Department of Management Science and Information Systems,
          Rutgers Business School Newark and New Brunswick,
          Rutgers University}
\affil[2]{\small Center for Applied Scientific Computing, Lawrence Livermore
          National Laboratory}
\affil[3]{\small Graduate School of Management, University of California
          Davis}

\date{\today}

\maketitle

\vspace{-3ex}

\begin{abstract}
We describe an asynchronous parallel method to calculate Lagrangian bounds for
mixed-integer stochastic programming problems.  As in some prior work, we use
a fully corrective Frank-Wolfe method on subproblems within a decomposition
method, eventually optimizing subproblems for individual scenarios over the
convex hulls of their integer-feasible vertices.  However, instead of using
progressive hedging as our decomposition algorithm, we employ an asynchronous
projective hedging method.  Like our earlier work on projective hedging, our
algorithm need only solve a subset of the scenario subproblems at each
iteration.  Here, we further exploit the capabilities of projective hedging to
create a parallel implementation capable of fully asynchronous
operation. As we demonstrate empirically, this capability can lead to significant
efficiency improvements when the subproblems are solved by Frank-Wolfe and
have highly variable solve times.
\end{abstract}

\section{Introduction}
\label{sec:introduction}
Boland \emph{et al.}~\cite{BCDELL17} described a method for computing tight
bounds for the optimal values of mixed-integer stochastic programming
problems, embedding a fully corrective Frank-Wolfe (FW)
algorithm~\cite{Hol74,vHoh77,LJJ15}, using a standard mixed-integer
programming solver as the linear minimization oracle (LMO) within a
progressive hedging (PH) decomposition procedure~\cite{RW91}.  Here, we
propose a similar algorithm, using the asynchronous projective hedging (APH)
method of~\cite{EWW25} in place of progressive hedging.  This combination has
two advantages: first, as in the computational experiments in~\cite{EWW25}, it
can take advantage of being able to solve only a subset of the scenario
subproblems at each iteration, which can in some cases lead to significant
performance improvements. The second advantage relates to the high
solution-time variability of subproblems involving MILPs.  If we were to
directly adapt our prior implementation approach from~\cite{EWW25} to use FW
and MILP-based subproblems, we would sometimes experience ``bottleneck''
periods in which a large fraction of the available parallel processing
resources would have to idle waiting for the slowest scenario MILPs to solve.
Instead, we use a new implementation technique described in
Section~\ref{sec:implementation_approach}, which constitutes a significant
advance over the implementation in~\cite{EWW25}, more fully exploiting the
asynchronous potential of the APH algorithm.  The computational experiments on
some standard ``SSLP'' test problems~\cite{NS05,NS08,SIPLIB} in
Section~\ref{sec:computational_testing} show how this approach essentially
removes bottlenecks, even though the underlying algorithm from~\cite{EWW25}
contains some synchronous elements.

\section{Problem setup, notation, and background}
\label{sec:problem_setup_and_notation}
We employ the style of stochastic programming formulation originating 
in~\cite{RW91}, using the following particular notation: suppose
that a stochastic programming problem is defined on a finite scenario tree
with stages indexed by $t \in 1..T$ and $\ell$ leaf nodes indexed by $i \in
1..\ell$.  For each node $N$ in the tree, let $t(N)$ denote its time stage.
Further, let $\mathcal{T}_t$ denote the set of nodes at time stage $t$ and let
$\mathcal{U} = \bigcup_{t=1}^{T-1} \mathcal{T}_t$ denote the set of non-leaf
nodes. For every node $N$, also let $\mathcal{D}(N)\subseteq 1..\ell$ denote the
set of leaf nodes that are descendants of $N$.  
For each $i\in 1..\ell$, denote the probability of leaf node $i$ by 
$\pi_i \in\; ]0,1]$, subject to $\sum_{i=1}^\ell \pi_i = 1$.  
For each node $N$ in the
tree, let $\pi(N) = \sum_{i\in\mathcal{D}(N)} \pi_i$ denote its probability.
Assume that all zero-probability nodes have already been pruned from the
scenario tree.  The term ``scenario'' will be taken to be synonymous with
``last-stage node'' and thus with the leaves of the tree.

For simplicity (and without loss of generality), assume that every stage-$t$
node in the tree has the same number of decision variables $n_t \geq 1$, and
let $\bar n \doteq \sum_{t=1}^T n_t$ denote the full number of decision
variables associated with each scenario and $\bar m \doteq \bar n - n_T$ be
similar, excluding the last stage.

The working variables $x$ form a grid consisting of a
subvector $x_{i} \in \real^{\bar n}$ of full root-to-leaf decision variables
for each scenario $i \in 1..\ell$; thus, their total dimension is $n \doteq
\ell \bar n$.  For each $i\in 1..\ell$ and $t\in 1..T$, let $x_{it} \in
\real^{n_t}$ denote the decision variables for scenario $i$, stage $t$, and
$x_{itj} \in \real$, $j\in 1..n_t$, denote the individual elements of this
vector. Define $\overline M$ to be the linear map that drops the last $n_T$
variables from a vector of length $\bar n$, dropping the last-stage variables
from a single scenario, that is, for any scenario $i \in 1..\ell$,
\begin{equation} \label{dropdef}
\overline Mx_{i} = \overline M (x_{i1},\ldots,x_{i,T-1},x_{iT}) 
               = (x_{i1},\ldots,x_{i,T-1}),
\end{equation}
so that
\begin{align*}
\overline{M} &= \big[ \identity_{\bar m} \;\; 0_{\bar m \times n_T}\big] &
&\text{and} &
\overline{M}\transpose &= 
\left[
\begin{array}{c}
\identity_{\bar m} \\ 0_{n_T \times\bar m }
\end{array}
\right],
\end{align*}
where $\identity_a$ denotes the $a \times a$ identity matrix and $0_{a \times
b}$ denotes the $a \times b$ zero matrix.

Let $M$ be the linear operator that, given a vector $x\in\real^n$ comprising a
subvector of length $\bar n$ for each scenario, drops the last-stage variables
from each subvector, so that
\begin{multline} \label{fulldropdef}
Mx = M\big(x_1,\ldots,x_\ell\big) 
   = M \left(
           \begin{array}{c}
             \big(x_{11},\ldots, x_{1,T-1}, x_{1T}\big) \\
             \vdots \\
             \big(x_{\ell 1}, \ldots, x_{\ell,T-1}, x_{\ell,T}\big)
           \end{array}
       \right) \\
   = \left(
           \begin{array}{c}
             \overline M \big(x_{11}, \ldots, x_{1,T-1}, x_{1T}\big) \\
             \vdots \\
             \overline M \big(x_{\ell 1},\ldots, x_{\ell,T-1}, x_{\ell,T}\big)
           \end{array}
       \right) 
   = \left(
           \begin{array}{c}
             \big(x_{11},\ldots, x_{1,T-1}\big) \\
             \vdots \\
             \big(x_{\ell 1},\ldots, x_{\ell,T-1}\big)
           \end{array}
       \right),
\end{multline}
or equivalently
\begin{align*}
M &=
\underbrace{
\left[
\begin{array}{cccc}
\overline{M} \\
& \overline{M} \\
& & \ddots \\
& & & \overline{M}
\end{array}
\right]
}_{\ell \text{~times}}
&
M\transpose &=
\underbrace{
\left[
\begin{array}{cccc}
\overline{M}\transpose \\
& \overline{M}\transpose \\
& & \ddots \\
& & & \overline{M}\transpose
\end{array}
\right]
}_{\ell \text{~times}}.
\end{align*}

The Rockafellar-Wets formulation approach dedicates multiple subvectors of a
decision vector $x\in\real^n$ to each node of the scenario tree; for example,
the root node corresponds to the $\ell$ subvectors $x_{11}, \ldots, x_{\ell
1}$.  For the solution corresponding to $x$ to be implementable for the
stochastic programming problem, all these subvectors must be equal,
$x_{11}=\cdots=x_{\ell 1}$, that is, the first-stage decisions cannot depend
on knowledge of which scenario will eventually transpire.  More generally, for
any non-leaf node $N \in \mathcal{U}$, all the corresponding stage-$t(N)$
decisions $x_{it(N)}$, for $i\in \mathcal{D}(N)$, must be identical, meaning
that the decisions at that node cannot depend on which descendant of $N$
eventually occurs. For vectors $z$ in the range space $\real^m$ of $M$, where $m \doteq \ell\bar m$, use
the same subvector indexing notation (excluding stage $T$), and define a
linear subspace $\mathcal{N} \subset
\real^m$ by
\begin{equation}
\mathcal{N} \doteq \set{z\in\real^m}{(\forall\,N\in\mathcal{U}) \;
                                      \big(\forall\,i,i' \in \mathcal{D}(N)\big) :
                                      z_{it(N)} = z_{i't(N)}}.
\end{equation}
Vectors in $z\in\mathcal{N}$ are called \emph{nonclairvoyant} or
\emph{nonanticipative} in the sense that $z_{it} = z_{i't}$ whenever scenarios
$i$ and $i'$ are indistinguishable at stage $t$.  They thus correspond to
implementable plans for stages $1..(T-1)$, requiring no knowledge of future
events.  Our particular interest is in mixed integer stochastic programs
formulated as
\begin{equation} \label{discreteSP}
\begin{array}{cll}
\displaystyle{\min_{x\in\real^n}} & \sum_{i=1}^\ell \pi_i h_i(x_i) \\
\suchthat & Mx \in \mathcal{N} \\
& x_i \in K_i &\forall\,i\in 1..\ell,
\end{array}
\end{equation}
where, for $i\in 1..\ell$, each $h_i : \real^{\bar n} \to \real$ is a
differentiable convex function and $K_i \subset \real^{\bar n}$ is a
(potentially very large) finite set.  In particular, we take each $K_i$ to be
the set of integer-feasible vertex solutions to a mixed-integer linear
constraint system.  The goal is to closely approximate (from below) the
optimal value of the relaxation
\begin{equation} \label{discreteSPconvexified}
\begin{array}{cll}
\displaystyle{\min_{x\in\real^n}} & \sum_{i=1}^\ell \pi_i h_i(x_i) \\
\suchthat & Mx \in \mathcal{N} \\
& x_i \in \conv K_i &\forall\,i\in 1..\ell,
\end{array}
\end{equation}
where ``$\conv$'' denotes the convex hull operation on a set. The optimal
objective value of~\eqref{discreteSPconvexified} is also the optimal value of
the Lagrangian dual problem
\begin{equation} \label{lagdual1}
\max_{p\in\mathcal{N}\orthog} \left\{
\min_{x \in K_1 \times \cdots \times K_{\ell}}
    \left\{
        p\transpose M x + \sum_{i=1}^{\ell} \pi_i h_i(x_i)
    \right\}
\right\},
\end{equation}
or equivalently
\begin{equation} \label{lagdual}
\max_{p\in\mathcal{N}\orthog} \left\{
\sum_{i=1}^\ell 
       \min_{x_i\in K_i} \big\{ \pi_i h_i(x_i) + p_i\transpose \overline{M} x_i \big\} \right\},
\end{equation}
where $p_i \in \real^{\bar m}$ denotes the subvector of the Lagrange multiplier
vector $p$ corresponding to scenario $i$.  To produce algorithms more closely
matching the original PH algorithm of~\cite{RW91}, it is customary to
use scaled Lagrange multipliers $w \in \real^m$ such that $p_i = \pi_i
w_i$ for each $i \in 1..\ell$.  This change of dual variables may be written
$p = \Pi w$ or $w = \Pi^{-1} p$, where $\Pi \doteq \diag(\pi_1 \identity_{\bar
m}, \pi_2 \identity_{\bar m}, \ldots, \pi_{\ell} \identity_{\bar m})$.
Defining the linear subspace $\mathcal{N}^* \doteq \set{w \in \real^m}{\Pi w
\in \mathcal{N}\orthog}$, an equivalent problem to~\eqref{lagdual} is
\begin{equation} \label{lagdualw}
\max_{w\in\mathcal{N}^*} \left\{
\sum_{i=1}^\ell 
       \pi_i \min_{x_i\in K_i} \big\{ h_i(x_i) + w_i\transpose \overline{M} x_i \big\} \right\}.
\end{equation}

In practice, a good approximation of the optimal (scaled) dual
variables $w^* \in \mathcal{N}^*$ optimizing~\eqref{lagdualw} has been shown
to enable generation of high-quality feasible solutions to~\eqref{discreteSP},
through various follow-on heuristics~\cite{mpisppy23}.

The classical approach to computing such Lagrangian bounds, originating with
the work of Shor~\cite{Shor62} and Ermoliev~\cite{Erm66}, is by applying a
subgradient method to the Lagrange multipliers~$p$.  However, it is difficult
to identify subgradient stepsizes that provide acceptable speed of
convergence, without danger of divergence.  The progressive hedging
algorithm~\cite{RW91} poses a tempting alternative, but simply imposing the
discrete constraints $x_i \in K_i$ in its subproblems violates its convergence
analysis' assumptions and has been shown in practice~\cite{WW11} to lead to
cycling and nonconvergent behavior.  More recently, Boland \emph{et
al.}~\cite{BCDELL17} proposed using PH, but employing an FW method to directly
solve its subproblems over the convexified sets $\conv K_i$. This approach is
in essence equivalent to applying PH to the convex stochastic programming
problem
\begin{equation}
\label{abstractSP}
\begin{array}{ll}
\displaystyle{\min_{x\in\real^n}}\; 
   & \displaystyle{\sum_{i=1}^{\ell} \pi_i f_i\big(x_{i}\big)} \\
\suchthat & Mx \in \mathcal{N}, \vphantom{\displaystyle{\sum_i}}
\end{array}
\end{equation}
where 
\begin{equation} \label{convexifiedfi}
(\forall\,i\in 1..\ell) \quad \quad
f_i(x_i) \doteq
\begin{cases}
h_i(x_i), \;\; & x_i \in \conv K_i \\
+\infty, & \text{otherwise.}
\end{cases}
\end{equation}
In this kind of application of FW, each LMO invocation consists of minimizing
the function $x_i \mapsto \inner{\nabla h_i(\bar x_i)}{x_i}$ over $x_i \in
K_i$ for each $i\in 1..\ell$, where $\bar x_i$ is the current FW iterate.
These minimizations are accomplished by calling a commercial MILP solver, with
each $K_i$ described by a collection of linear and integrality constraints.
We refer to the resulting combined method as ``PH-FW.''

\section{Frank-Wolfe specifics}

To prove convergence of PH-FW, Boland \emph{et al.} solve the resulting
scenario subproblem with the ``fully corrective'' variant of the FW
algorithm~\cite{Hol74,vHoh77,LJJ15}, also called ``simplicial decomposition''
and here abbreviated ``FCFW.'' This class of FW method keeps track of the set
$R$ of LMO solutions obtained so far and obtains its next iterate not by a
step in the direction of the LMO solution, but by solving a continuous problem
optimizing the full nonlinear objective function over $\conv R$.  Our goal
here is to implement a similar method, but replacing PH with the asynchronous
projective hedging (APH) method of~\cite{EWW25} to produce an algorithm more
amenable to asynchronous parallel implementation. We refer to this combination
as ``APH-FW.''

Our current convergence proof for APH-FW is similar in basic principle to the
PH-FW convergence proof in~\cite{BCDELL17}, so we also use a fully corrective
FW method, alternating between linear LMO solves and (essentially) exact
nonlinear solves over the set of vertices collected so far.  However, we use a
``nonlinear-first'' variant that starts this alternation with a nonlinear
solve, whereas~\cite{BCDELL17} uses a ``linear-first'' variant beginning with an
LMO call.  One reason~\cite{BCDELL17} selects the linear-first approach is that,
by choosing the FW starting points for the scenarios $i\in 1..\ell$ 
so that they match the current Lagrange multipliers $w^k$ at the
start of PH iteration $k$, each PH iteration $k$ produces as a
``side effect'' a lower bound on the optimal value of the convex
relaxation~\eqref{discreteSPconvexified} and thus a lower bound on the optimal
value of the mixed-integer stochastic program~\eqref{discreteSP}. This lower
bound is formed by combining the results of each scenario $i$'s first LMO call
in that PH iteration $k$.  

In an alternative, asynchronous implementation
based on the algorithm in~\cite{EWW25}, 
one cannot readily assemble a collection of LMO results
over all scenarios $i \in 1..\ell$ that were computed from the same Lagrange
multiplier estimate $w^k$, because some scenarios $i$ will typically have
their most recent subproblem objective functions derived from older Lagrange
multiplier estimates than others.  Therefore, an immediate sequence of
Lagrangian bounds on the optimal objective value cannot be a low-cost
byproduct of the algorithm, no matter which variant of FW is used.  This
benefit being unavailable, we instead choose the nonlinear-first approach,
because it permits some simplification in the algorithm analysis, and removes
the need for a technical recourse assumption used in~\cite{BCDELL17}, as will
be discussed further below.  We therefore develop a nonlinear-first variant of
FCFW, which is a modest modification to currently described versions, along
with some associated analysis. The first steps in this analysis are some
elementary results regarding when a point is optimal over $\conv K$, for $K
\subset \real^d$, and, for some ``restriction'' $R \subseteq K$ of $K$,
relating optimality over $\conv R$ to optimality over $\conv K$.

\begin{lemma} \label{lem:optCond}
Let $f:\real^d\to\real$ be convex and continuously differentiable, and let
$K \subset \real^d$ be nonempty.  Then a necessary and sufficient condition for
$x^*\in\real^d$ to minimize $f$ over $\conv K$ is
\begin{equation} \label{optK}
(\forall\,x\in K) \quad \inner{\nabla f(x^*)}{x^*} \leq \inner{\nabla f(x^*)}{x}.
\end{equation}
\end{lemma}
\begin{proof}
Using elementary convex analysis, a necessary and sufficient condition for
$x^*$ to minimize $f$ over the convex set $\conv K$ is that $0 \in \nabla
f(x^*) + N_{(\conv K)}(x^*)$, where $N_X$ denotes the normal cone map of the set
$X$. This condition is equivalent to $-\nabla f(x^*) \in N_{(\conv K)}(x^*)$,
which is in turn equivalent to $\nabla f(x^*)$ making a nonnegative inner
product with every feasible direction from $x^*$ within $\conv K$, that is,
\begin{align}
&& (\forall\,x\in \conv K) && \inner{\nabla f(x^*)}{x - x^*} &\geq 0 
   &&&& \label{feasdircond} \\
\Leftrightarrow && (\forall\,x\in \conv K) &&
   \inner{\nabla f(x^*)}{x^*} &\leq \inner{\nabla f(x^*)}{x}.  \label{vareq}
\end{align}
Now, \eqref{optK} is clearly implied by~\eqref{vareq}, because they are
identical except for the ``for all'' quantifier in~\eqref{vareq} ranging over
the larger set $\conv K \supseteq K$.  Thus, \eqref{optK} is necessary for optimality.

Regarding sufficiency, suppose~\eqref{optK} holds and let $x$ be any element
of $\conv K$.  Then $x$ must be expressible as a positive convex combination
of at most $d+1$ points in $K$, that is, for some $\bar{d} \in 1..(d+1)$,
\begin{align*}
x &= \sum_{i=1}^{\bar{d}} \alpha_i x^i, & \text{where} &&
\alpha_1,\ldots,\alpha_{\bar{d}} &> 0 &
x^1,\ldots,x^{\bar{d}} &\in K&
\sum_{i=1}^{\bar{d}} \alpha_i = 1.
\end{align*}
For each $i \in \{1,\ldots,\bar{d}\}$, one has $x^i\in K$ and therefore
from~\eqref{optK} that
\begin{align*}
\inner{\nabla f(x^*)}{x^*} \leq \inner{\nabla f(x^*)}{x^i}.
\end{align*}
Multiplying each of these inequalities by the respective $\alpha_i > 0$ and
summing, one has
\begin{align*}
&&
\sum_{i=1}^{\bar{d}} \alpha_i \inner{\nabla f(x^*)}{x^*} &\leq 
\sum_{i=1}^{\bar{d}} \alpha_i \inner{\nabla f(x^*)}{x^i} &&&& \\
\Leftrightarrow &&
\inner{\nabla f(x^*)}{x^*} 
   &\leq \Inner{\nabla f(x^*)}{\sum_{i=1}^{\bar{d}} \alpha_i x^i} \\
\Leftrightarrow &&
\inner{\nabla f(x^*)}{x^*} &\leq \inner{\nabla f(x^*)}{x},
\end{align*}
where the first equivalence uses that $\sum_{i=1}^{\bar{d}} \alpha_i = 1$ and the
second uses that $x = \sum_{i=1}^{\bar{d}} \alpha_i x^i$.  Since $x\in\conv K$ was
arbitrary, it follows that~\eqref{optK} implies~\eqref{vareq} and is thus a
sufficient condition for optimality.
\end{proof}

\begin{lemma} \label{lem:terminate}
Let $f:\real^d\to\real$ be convex and continuously differentiable and
$\emptyset \neq R \subset K \subset \real^d$.  Suppose that $\bar
x$ minimizes $f$ over $\conv R$ and $\hat x$ minimizes $\inner{\nabla f(\bar
x)}{x}$ over $x \in K$.  If $\hat x \in \conv R$, then $\bar x$ minimizes $f$
over all of $\conv K$.
\end{lemma}
\begin{proof}
Having $\hat x$ minimize $\inner{\nabla f(\bar x)}{x}$ over $K$ means that 
\begin{equation} \label{optlinear}
(\forall\,x\in K) \quad 
   \inner{\nabla f(\bar x)}{\hat x} \leq \inner{\nabla f(\bar x)}{x}, 
\end{equation}
while having $\bar x$ minimize $f$ over $\conv R$ means that $\inner{\nabla
f(\bar x)}{x - \bar x} \geq 0$ for all $x \in \conv R$, that is,
\begin{equation}
(\forall\,x\in \conv R) \quad
   \inner{\nabla f(\bar x)}{\bar x} \leq \inner{\nabla f(\bar x)}{x}.  
       \label{restrictedvar}
\end{equation}
Now, if $\hat x \in \conv R$, \eqref{restrictedvar} implies that 
$\inner{\nabla f(\bar x)}{\bar x} \leq \inner{\nabla f(\bar x)}{\hat
x}$, which may be combined with~\eqref{optlinear} to yield
\begin{equation} \label{optcondfull}
(\forall\,x\in K) \quad 
   \inner{\nabla f(\bar x)}{\bar x} \leq \inner{\nabla f(\bar x)}{\hat x} 
                                    \leq \inner{\nabla f(\bar x)}{x}.
\end{equation}
By Lemma~\ref{lem:optCond}, this condition is sufficient for
$\bar x$ to minimize $f$ over all of $\conv K$.
\end{proof}

The next result provides a way of bounding the suboptimality of the current
iterate in FW algorithms.  It is standard and therefore stated
without proof.

\begin{lemma} \label{lem:fwGap}
Suppose that $\emptyset \neq K \subset \real^d$ and $f:\real^d \to \real$ is a
continuously differentiable convex function.  Further,
suppose that $\hat x \in \argmin_{x\in K} \inner{\nabla f(\bar x)}{x}$. Then
\begin{equation} \label{fwgap}
f(\bar x) \leq \inf_{y\in\conv K}\big\{ f(y) \big\} 
             + \inner{\nabla f(\bar x)}{\bar x - \hat x}.
\end{equation}
The quantity $\gamma \doteq \inner{\nabla f(\bar x)}{\bar x - \hat x}$ on the
right side of~\eqref{fwgap} must be nonnegative.  If $\gamma=0$, then $\bar x$
minimizes $f$ over $\conv K$.  Conversely, $\bar x$ minimizing $f$ over $\conv
K$ implies that $\gamma =
\inner{\nabla f(\bar x)}{\bar x - \hat x} = 0$.
\end{lemma}

The scalar $\gamma = \inner{\nabla f(\bar x)}{\bar x - \hat x}$ in this result
is referred to as the \emph{Frank-Wolfe gap} and dates back
to~\cite{FW56,DW61,Mang69} and can be used to assess the suboptimality of any
point $\bar x$ over the convex hull of any set $K$ for which a linear
minimization oracle is available.

FCFW-class algorithms for optimizing $f$ over $\conv K$ are organized around
the following fundamental operations:
\begin{enumerate}
\item Minimizing $f$ over $\conv R$, where $R\subseteq K$
\item \label{operationType:linear} Minimizing $\inner{\nabla f(\bar x)}{x}$ over
$x\in K$, where $\bar x \in \conv R$
\item Adjoining solution points obtained from the linear minimizations
(operations of type~\ref{operationType:linear}) to~$R$.
\end{enumerate}
Combining the above lemmas, the following simple result shows that a zero
Frank-Wolfe gap is necessary and sufficient to detect when a point $\bar x$
minimizing $f$ over $\conv R$ also minimizes $f$ over the larger set $\conv
K$:

\begin{lemma} \label{lem:fwGapTerm}
For $f:\real^d\to\real$ convex and continuously differentiable, suppose that
\begin{align*}
\emptyset &\neq R \subseteq K \subset \real^d &
\bar x &\in \conv R &
\hat x &\in \argmin_{x \in K} \inner{\nabla f(\bar x)}{x}.
\end{align*}
Then $\bar x$ minimizes $f$ over all of $\conv K$ if and
only if $\inner{\nabla f(\bar x)}{\bar x - \hat x} = 0$.
\end{lemma}
\begin{proof}
Since $R \subseteq K$, one has $\conv R \subseteq \conv K$ and therefore $\bar
x \in \conv K$.  So the last statement of Lemma~\ref{lem:fwGap} states that
$\bar x$ minimizes $f$ over $\conv K$ if and only if $\inner{\nabla f(\bar
x)}{\bar x - \hat x} = 0$.
\end{proof}
This result makes it attractive to terminate the kind of algorithms under
discussion based on the Frank-Wolfe gap rather than by testing membership of
$\hat x \in \argmin_{x \in K} \inner{\nabla f(\bar x)}{x}$ in $\conv R$ or~$R$,
both of which would typically require more burdensome computations.  Of course, numerical
precision issues dictate using some tolerance rather than requiring that the gap be
exactly zero, but such issues would also apply to $\conv R$ or $R$ membership tests.
Another simple consequence of the above lemmas is that whenever the membership
test would indicate termination, the gap test will do so as well (not
considering numerical round-off issues):

\begin{lemma} \label{lem:incMeansGap}
For $f:\real^d\to\real$ convex and continuously differentiable, suppose that
\begin{align*}
\emptyset &\neq R \subseteq K \subset \real^d &
\bar x &\in \argmin_{x\in\conv R} \big\{ f(x) \big\} &
\hat x &\in \argmin_{x \in K} \inner{\nabla f(\bar x)}{x} &
\hat x &\in \conv R,
\end{align*}
that is, minimizing $\inner{\nabla f(\bar x)}{x}$ returns a point already in $\conv
R$. Then $\inner{\nabla f(\bar x)}{\bar x - \hat x} = 0$.
\end{lemma}
\begin{proof}
In the hypothesized situation, Lemma~\ref{lem:terminate} asserts that $\bar x$
minimizes $f$ over $\conv K$.  Lemma~\ref{lem:fwGapTerm} then
guarantees that $\inner{\nabla f(\bar x)}{\bar x - \hat x} = 0$.
\end{proof}

\noindent The above results permit analysis of the following minor variation
on the FCFW method for minimizing a convex differentiable function $f :
\real^d \to \real$ over the convex hull of $K \subset \real^d$:

\begin{enumerate}[start=0]
\item Start with $j=1$ and any nonempty $R_1 \subseteq K$.
\item \label{item:nonlinFirst:nonlin} Find $\bar x^j \in \argmin_{x\in\conv R_j}
\big\{ f(x) \big\}$.
\item \label{item:nonlinFirst:lin} Find $\hat x^j\in\argmin_{x\in K} \big\{
\inner{\nabla f(\bar x^j)}{x} \big\}$.
\item \label{item:nonlinFirst:term}
If $\inner{\nabla f(\bar x^j)}{\bar x^j - \hat x^j} \leq \tau$,
where $\tau \geq 0$ is some tolerance, return the solution $\bar x^j$
(an approximate solution if $\tau > 0$).
\item \label{item:nonlinFirst:adjoin} Otherwise, adjoin $R_{j+1} \leftarrow
R_j \cup \{\hat x^j\}$, set $j \leftarrow j + 1$, and return to
step~\ref{item:nonlinFirst:nonlin}.
\end{enumerate}

This nonlinear-first method is initialized slightly differently than the more
common, linear-first version of FCFW, with at least one point from $K$, rather
than with an initial solution guess $\bar x^1 \in \conv K$ as in the case of
the ``linear first'' variant.  With the aid of the above lemmas, one can analyze
its behavior:

\begin{proposition} \label{prop:nonlinFirst}
The nonlinear-first FCFW algorithm described immediately above has the
following properties:
\begin{itemize}[nosep]
\item At every iteration $j$ at which the method does not terminate,
$\card{R_{j+1}} = \card{R_j} + 1$.
\item If $K$ is finite, the algorithm must terminate no later than iteration
$\card{K} - \card{R_1} + 1 \leq \card{K}$.
\item The point $\bar x^j$ returned by the algorithm has $f(\bar x^j) \leq
\inf_{x\in\conv K}\big\{f(x)\big\} + \tau$, that is, it minimizes $f$ over
$\conv K$ within an objective function error of $\tau$.
\end{itemize}
\end{proposition}
\begin{proof}
Every iteration $j$ of this algorithm sets $\bar x^j \in \argmin_{x\in\conv R_j}
\big\{ f(x) \big\}$ in step~\ref{item:nonlinFirst:nonlin}.  Therefore if
step~\ref{item:nonlinFirst:lin} of any iteration $j$ produces a solution $\hat
x^j$ already in $R_j$, the hypotheses of Lemma~\ref{lem:incMeansGap} are
satisfied with $\bar x = \bar x^j$, $\hat x = \hat x^j$, and $R=R_j$, so the
Frank-Wolfe gap should be zero and the subsequent execution of
step~\ref{item:nonlinFirst:term} should immediately halt the algorithm.
Therefore, if an iteration~$j$ reaches step~\ref{item:nonlinFirst:adjoin},
then it must have been the case that
$\hat x^j \not\in R_j$ and therefore $\card{R_{j+1}} = \card{R_j} + 1$, establishing
the first claim.

Now suppose that $K$ is finite and consider the state of the algorithm at the
beginning of any iteration ${\bar j} \geq 1$.  Every one of the ${\bar j} - 1$
prior iterations $j$ (if any) has added one new point to $R_j$, so
$\card{R_{\bar j}} = \card{R_1} + {\bar j} - 1$. Solving $\card{K} =
\card{R_1} + {\bar j} - 1$ for ${\bar j}$ yields ${\bar j} = \card{K} -
\card{R_1} + 1$. At that iteration, one must have $\card{R_{\bar j}} =
\card{K}$, hence $R_{\bar j} = K$, and step~\ref{item:nonlinFirst:lin} cannot
avoid generating a point $\hat x^{\bar j}$ that is already in $R_{\bar j}$.
Then, as shown at the beginning of this proof, the method will immediately
terminate in step~\ref{item:nonlinFirst:term}.  Thus, the claim that the
algorithm can run at most $\card{K} - \card{R_1} + 1$ iterations has been
established.  Since $R_1$ is assumed nonempty, this number of iterations can
be no more than $\card{K}$, so the second claim is now fully proved.

The final assertion, regarding the optimality of the returned solution $\bar
x^j$, follows immediately from Lemma~\ref{lem:fwGap}.
\end{proof}

\noindent The above analysis resembles that of the linear-first variant
(omitted since that variant of the method is not used here), but is slightly
simpler.

\section{The APH stochastic programming method}

\begin{algorithm}[t]
\label{alg:aph}
\DontPrintSemicolon
\SetCommentSty{rmfamily}
\caption{Asynchronous projective hedging algorithm 
         from~\cite{EWW25}\rule{0pt}{3ex}.}
Start with any $z^0 \in \mathcal{N}$, $w^0 \in \mathcal{N}^*$
   and $\big(x_i^{-1}, y_i^{-1} \in \real^{\bar m}, x_{iT}^{-1} \in \real^{n_T}\big)_{i\in 1..\ell}$  \;
\For{$k=0,1,\ldots$}{
Determine some $I_k \subseteq 1..\ell$  
      \label{line:pickblocks} \;
\For{$i\in I_k$}{
  \tcp{Incorporate results of a solved subproblem} 
  For some $d(i,k) \in 0..k$, set $(x_i^k,x_{iT}^k,y_i^k)
        = \mathcal{S}(i,z_i^{d(i,k)}, w_i^{d(i,k)})$  
      \label{line:incorporate}
} 
\lFor{$i \not\in I_k$}
     {$(x_i^k,x_{iT}^k,y_i^k) = (x_i^{k-1},x_{iT}^{k-1},y_i^{k-1})$} 
      \label{line:rollover}
$u^k = \proj^{(\pi)}_{\mathcal{N}^*}(x_1^k,\ldots,x_\ell^k)$  \label{line:uproj}   \;
$v^k = \proj^{(\pi)}_{\mathcal{N}}(y_1^k,\ldots,y_\ell^k)$    \label{line:vproj}   \;
$\eta_k = \sum_{i=1}^\ell \pi_i\big(\norm{u_i^k}^2 + \xi \norm{v_i^k}^2\big)$ 
      \label{line:phlikedenom}\; 
$\varphi_k = \sum_{i=1}^{\ell} \pi_i (z_i^k-x^k_i)\transpose(w^k_i-y_i^k)$
      \label{line:phlikenumer}\;
For some $\nu_k \in (0,2)$, set
$\theta_k = {\nu_k} \max\left\{0, \varphi_k \right\}/{\eta_k}$ 
      \label{line:phlikestepsize} \;
$z^{k+1} = z^k + \theta_k \xi v^k$ \label{line:phlikezupdate} \;
$w^{k+1} = w^k + \theta_k u^k$ \label{line:phlikewupdate}
} 
\end{algorithm}

Algorithm~\ref{alg:aph} presents the asynchronous projective hedging (APH)
method from~\cite{EWW25} in notation adapted to this paper, solving convex
stochastic programming problems expressed in the
form~\eqref{abstractSP}.  To simplify reasoning about the algorithm, it is
presented as if it is serial; any parallel implementation that has the same
effect will of course be subject to the same convergence analysis.  Noteworthy
points about the algorithm include:

\begin{enumerate}
    \item For each $i\in 1..\ell$ the working variables $x_i^k$ and $y_i^k$
    lie in $\real^{\bar m}$, with the dimension of a single scenario without
    the last-stage variables.  The working variables $x_{iT}^k \in
    \real^{n_T}$ keep track of the last-stage primal variables for scenario
    $i$.

    \item The parameter $\xi$, corresponding to $\gamma^{-1}$ in~\cite{EWW25}, is an
    arbitrary fixed positive scalar controlling the relative weighting of
    primal and dual optimality.

    \item The placeholder function $\mathcal{S}$ on
    line~\ref{line:incorporate} denotes the results of solving a subproblem
    for scenario $i\in 1..\ell$.  In the original implementation
    in~\cite{EWW25}, the operation $\mathcal{S}(i,z_i^{d(i,k)}, w_i^{d(i,k)})$
    takes the form
    \begin{align}
    &\text{Receive~} i, w_i \leftarrow w_i^{d(i,k)}, 
                        z_i \leftarrow z_i^{d(i,k)} \text{~from caller}
                        \label{startstandardsubproblem} \\
    &\text{Select some~} \rho_{ik} \in [\underline{\rho},\overline{\rho}]\\
    &\text{Find~} x_i \in \argmin_{x_i\in\real^{\bar n}} 
    \left\{ f_i(x_i) + w_i\transpose\overline{M}x_i
            + \frac{\rho_{ik}}{2}\norm{\overline{M}x_i - z_i}^2 \right\} 
         \label{scenariomin} \\
     &y_i \leftarrow w_i + \rho_{ik}(\overline{M}x_i - z_i) \label{yupdate} \\
     &\text{Return} \quad (\overline{M}x_i, x_{iT}, y_i),
         \label{endstandardsubproblem}
    \end{align}
    where $0 < \underline{\rho} 
    \leq \overline \rho < \infty$.  In the Lagrangian bound computation
    procedure presented below, we will substitute an FCFW procedure
    for~\eqref{startstandardsubproblem}-\eqref{endstandardsubproblem}.

    \item For brevity, a few details are omitted from the description.  In
    particular, it omits detection of the condition $\eta_k = 0$ which would
    render the calculation on line~\ref{line:phlikestepsize} undefined. Our
    implementation terminates whenever $\eta_k$ falls below a small, user-specified
    threshold, since it can be shown that a zero value implies
    optimality~\cite{EWW25,Eck17}.
    The full details of deciding
    when to terminate the outer ($k$) loop are also omitted.

    \item The operator $\proj^{(\pi)}_{\mathcal{N}}$ on line~\ref{line:vproj}
    performs a projection on the nonclairvoyant subspace $\mathcal{N}$, using
    a metric weighted by the scenario probabilities $\pi$.  At every non-leaf
    tree node $N \in \mathcal{U}$, this projection is accomplished by (using
    the same indexing conventions for $v$ as for the other vectors in
    $\real^m$)
    \begin{equation} \label{phprojformula}
    (\forall\,N\in \mathcal{U}) \;
    \big(\forall\,i\in \mathcal{D}(N)\big) \;
    (\forall\,j\in 1..n_{t(N)}) \quad
    v_{it(N)j}^{k} = \frac{1}{\pi(N)}
       \sum_{i'\in \mathcal{D}(N)} \pi_{i'} y_{i't(N)j}^{k}.
    \end{equation}
    That is, at each non-leaf node $N \in \mathcal{U{}}$, the vectors
    $v_{it(N)}^{k}, i\in\mathcal{D}(N)$ are obtained by taking the
    $\pi$-weighted average, across the scenarios in $\mathcal{D}(N)$, of the
    corresponding elements in $y^{k}$.  The calculation of $u^{k}$ in
    line~\ref{line:uproj} obeys the equations
    \begin{equation} \label{conjproj}
    u^{k} = \proj^{(\pi)}_{\mathcal{N}^*}(x^k) 
            = (\identity - \proj^{(\pi)}_{\mathcal{N}})(x^k)
            = x^k - \proj^{(\pi)}_{\mathcal{N}}(x^k),
    \end{equation}
    and so may be implemented by essentially the same calculation.  These
    computations may be performed efficiently in a parallel setting, as
    in~\cite{mpisppy23,EWW25}.  The remaining calculations on
    lines~\ref{line:phlikedenom}-\ref{line:phlikewupdate} also have efficient
    parallel implementations.

    \item In an asynchronous parallel implementation, the sets $I_k$ of
    scenarios whose results are incorporated into the calculation at iteration
    $k$ are not typically explicitly chosen at iteration $k$.  Instead, the
    scenario calculation for a scenario $i \in I_k$ may have been initiated at
    an earlier iteration $d(i,k) \leq k$, becoming complete between iterations
    $k-1$ and $k$.  

    \item The algorithm is based on the asynchronous abstract projective
    splitting method for monotone operators in~\cite{Eck17}, which is a
    simplified version of the method proposed in~\cite{ComEck15}.  The
    calculations in steps~\ref{line:uproj}-\ref{line:phlikewupdate} project
    the previous estimate $(z^k,w^k)$ of the joint primal-dual solution onto a
    hyperplane separating it from the optimal primal-dual set.
\end{enumerate}

\noindent Convergence of the method is established in the following result,
proved in~\cite{EWW25}:

\begin{proposition} \label{prop:phlikeconverge}
Suppose that a stochastic programming problem formulated
as~\eqref{abstractSP} has a solution and at least one of the following holds:
\begin{enumerate}
\item \label{assump:linear} The functions $f_1,\ldots,f_\ell$ are polyhedral, or
\item \label{assump:slater} $\mathcal{N} \cap \big( \bigtimes_{i=1}^{\ell} \ri
\overline M(\dom f_i) \big) \neq \emptyset$, where the operator $\ri$ denotes
taking the relative interior of a set as defined for example in~\cite[Section
1.4]{BerConvex}.
\end{enumerate}
Also suppose that Algorithm~\ref{alg:aph} operates in accordance with
the following conditions:
\begin{enumerate}
\item There exist $\underline{\nu},\bar{\nu}\in
  \;]0,2[$ such that $\underline{\nu} \leq \nu_k \leq \overline{\nu}$
  for all $k$.
\item The scenario calculation $\mathcal{S}$ is performed as
in~\eqref{startstandardsubproblem}-\eqref{endstandardsubproblem} for some
$\underline{\rho} > 0$ and $\overline{\rho} \in [\underline{\rho},\infty)$
that remain fixed throughout the algorithm.
\item \label{item:quasicyclicagain} For some fixed integer $B \geq 1$, 
$\left(\bigcup_{k=j}^{j + B -1} I_k \right) = 1..\ell$ 
for all $j \geq 0$.
\item \label{item:delayagain} For some fixed integer $D \geq 0$, we have
  $k - d(i,k) \leq D$ for all $i,k$ with $i \in I_k$.
\end{enumerate} 
Then
\begin{align} \label{phlikeconvergence}
z^k &\rightarrow x^\infty &
x^k &\rightarrow x^\infty &
w^k &\rightarrow w^\infty &
v^k &\rightarrow 0 &
u^k &\rightarrow 0,
\end{align}
where $x^\infty = M\bar{x}^*$ for some optimal solution $\bar{x}^*\in\real^n$
of~\eqref{abstractSP}, that is, $x^\infty$ is some optimal solution
to~\eqref{abstractSP} with the last-stage variables deleted, and $w^\infty =
\Pi^{-1} p^*$ for some optimal solution $p^*$ to the Lagrangian dual
problem~\eqref{lagdual}. Furthermore, all limit points (if any) of the
sequence
\begin{equation*}
\Big\{\big((x_i^k,x_{iT}^k) \big)_{i=1}^\ell \Big\}_{k=0}^{\infty}
\end{equation*}
are optimal solutions of~\eqref{abstractSP}.
\end{proposition}

Algorithm~\ref{alg:aph} simultaneously supports two forms of asynchrony, which may 
respectively be called
``block'' asynchrony and ``lag'' asynchrony.  Block asynchrony, sometimes
called ``block activation,'' is not asynchrony in the colloquial sense,
but refers only 
to having $I_k \neq 1..\ell$ for at least some~$k$, that is, processing only a subset
of the scenarios at some or all iterations. If the algorithm is operating in a
purely block-asynchronous mode, then $d(i,k) = k$ for all $k \geq 0$ and $i
\in I_k$.  In an implementation of Algorithm~\ref{alg:aph} supporting only block
asynchrony, each iteration directly selects a set of scenarios $I_k$ to
process, initiates their subproblem solves, waits until these subproblem
calculations have all completed, and then proceeds to the projection
calculations on lines~\ref{line:uproj}-\ref{line:phlikewupdate}.

Lag asynchrony, on the other hand, refers to having $d(i,k) < k$ for some $k
\geq 0$ and $i \in I_k$.  For instance, when the iteration counter $k$ 
has some value $k_1$, a parallel implementation might initiate processing of a
subproblem $i$ and, without waiting for $i$'s results, proceed to other
calculations.  Shortly thereafter, for example, some other scenario $i' \neq
i$ might finish, and the algorithm could use the information from $i'$ (and
possibly other recently completed scenarios) to update the Lagrange multiplier
and nonclairvoyant solution estimates (``centers'') to new respective values
$w^{k_1 + 1}$ and $z^{k_1 + 1}$. Eventually, at some iteration $k_2 \geq k_1 +
1$, processing for scenario $i$ will complete and be used (possibly along with
other scenarios) to update the multipliers and centers.  One then has
$d(i,k_2) = k_1$, the iteration number at which the just-completed scenario
solve for $i$ was initiated.

Block asynchrony was the only kind tested in~\cite{EWW25}, for two reasons:
\begin{enumerate}
\item The experiments in~\cite{EWW25} involved only continuous convex
problems, whose scenario solution times had very low
variability.  In such cases, pure block asynchrony should work efficiently,
because hardly any time needs to be spent idling between the solution of the
quickest scenario in $I_k$ and the slowest scenario in $I_k$, and there would
likely be little benefit to moving to full (block combined with lag)
asynchrony.

\item The \texttt{mpi-sppy} software platform~\cite{mpisppy23} used for the tests
in~\cite{EWW25} was originally designed for the standard, fully synchronous PH
algorithm.  While it was readily modified to support block asynchrony,
the platform lacked suitable infrastructure to implement lag asynchrony.
\end{enumerate}

While Algorithm~\ref{alg:aph} supports both block and lag asynchrony, it still
has synchronous elements, in that the coordination operations on
lines~\ref{line:uproj}-\ref{line:phlikewupdate} operate on the working
variables in a fully coordinated manner, much like PH but with some added
complexity.  The method is asynchronous only in that the solution of
subproblems, whose information is incorporated into the main loop on
line~\ref{line:incorporate}, may be de-synchronized from the central updating
process to a degree bounded by the parameters $B$ and $D$ in
Proposition~\ref{prop:phlikeconverge}.  Specifically, $B$ bounds how
unbalanced the algorithm can be in paying attention to the various scenarios,
and $D$ bounds the maximum iterations of delay between initiating a subproblem
solve and incorporating its results into the iterate path. Neither $B$ nor $D$
appears in nor constrains any of the update formulas in the algorithm; these
parameters instead only appear within the convergence analysis of the abstract
projective splitting algorithm in~\cite{ComEck15,Eck17} upon which the
algorithm is based.  By contrast, the asynchronous version of PH proposed
in~\cite{BLetal20}, based on the abstract splitting method in~\cite{ARock},
must underrelax some algorithmic steps by factors related to the equivalent of
$D$.  The method proposed in~\cite{BLetal20} may be viewed as ``more
asynchronous'' than Algorithm~\ref{alg:aph} in that incorporating the results
of a single scenario solve leaves some of the working variables of other
scenarios unchanged (when the stochastic program has more than two time
stages).  However, as we will show in
Section~\ref{sec:implementation_approach}, Algorithm~\ref{alg:aph} is
``asynchronous enough'' that a careful implementation using current HPC
hardware and software tools can achieve a very high level of asynchrony, with
the advantage that theoretical convergence is not dependent on underrelaxation
factors whose values in turn require knowing bounds on the time behavior of
the implementation.

\section{Combined algorithm}
To compute Lagrangian bounds for stochastic integer programming problems, we
now adapt Algorithm~\ref{alg:aph} to problems of the
form~\eqref{discreteSPconvexified} by replacing the subproblem processor
$\mathcal{S}$ specified
in~\eqref{startstandardsubproblem}-\eqref{endstandardsubproblem} with one
using nonlinear-first FCFW.  This scenario subroutine is shown in
Algorithm~\ref{alg:nffcfw}.

~

\begin{algorithm}[ht]
\label{alg:nffcfw}
\SetKw{KwTo}{to}
\SetKw{KwBreak}{break}
\SetCommentSty{rmfamily}
\DontPrintSemicolon
\caption{\rule{0pt}{3ex}Subproblem solution procedure $\mathcal{S}$ for
applying FCFW within solution of~\eqref{discreteSPconvexified}.} 
Receive $i, w_i \leftarrow w_i^{d(i,k)}, 
            z_i \leftarrow z_i^{d(i,k)}$ from caller \;
Select some $\rho_{ik} \in [\underline{\rho},\overline{\rho}]$ \;
Let $h_{ik}$ denote the function 
    $x_i \mapsto f_i(x_i) + w_i\transpose\overline{M}x_i
            + \frac{\rho_{ik}}{2}\norm{\overline{M}x_i - z_i}^2$ \;
\For{$l=0,1,\ldots$}{
    $\bar x_i \leftarrow \displaystyle{\argmin_{x_i \in \conv R_i} 
                           \big\{ h_{ik}(x_i) \big\}}$ 
                           \tcp*[r]{Using a nonlinear convex solver} 
    \lIf{$l \geq l_{\max}$}{\KwBreak} \label{line:hitlmax}
    $\hat x_i \leftarrow 
            \displaystyle{\argmin_{x_i \in K_i} 
                \biginner{\nabla h_{ik}(\bar x_i)}{x_i}}$ \label{line:LMOcall}
            \tcp*[r]{Using a MILP solver} 
    \lIf{$\biginner{\nabla h_{ik}(\bar x_i)}
                   {\bar x_i - \hat x_i} \leq \tau$}{\KwBreak} \label{line:FWbreakout}
    $R_i \leftarrow R_i \cup \{\hat x_i\}$ \label{line:adjoinVertex}
}
$y_i \leftarrow w_i + \rho_{ik}(\overline{M}\bar x_i - z_i)$ \;
\Return{$(\overline{M}\bar x_i, \bar x_{iT}, y_i)$}
\end{algorithm}

~

In Algorithm~\ref{alg:nffcfw}, the tolerance $\tau$ is a small nonnegative
scalar, $l_{\max}$ is a positive integer, and the sets $\{R_i\}_{i=1}^{\ell}$
collect the vertices of the respective sets $\{K_i\}_{i=1}^{\ell}$ encountered
so far. Line~\ref{line:adjoinVertex} expands these sets whenever a new vertex
is encountered.  The $R_i$ persist across calls to the subproblem processor, so
their size increases over time.  The algorithm must initialize each $R_i$ with
at least one element of $\conv K_i$, a straightforward task little different
from the initialization required in the PH-FW algorithm of~\cite{BCDELL17}.

Whenever the LMO call on line~\ref{line:LMOcall} returns a point $\hat x_i$
already in $R_i$ and hence also in $\conv R_i$, then Lemma~\ref{lem:incMeansGap}
asserts that the FW gap $\biginner{\nabla h_{ik}(\bar x_i)}{\bar x_i - \hat
x_i}$ should be zero, meaning that line~\ref{line:FWbreakout} will immediately
terminate the FW ($l$) loop and allow the subroutine to return with an
essentially exact solution.  After the MILP solver has been called $l_{\max}$
times, however, the test on line~\ref{line:hitlmax} will cause the
FW loop to exit, returning a solution of indeterminate quality.

We now state and prove a convergence result for the combined algorithm, using
similar arguments to those in the main convergence result (Proposition 3.3)
in~\cite{BCDELL17}.  Its assumption $\tau=0$ is the same as
in~\cite{BCDELL17}, but the proof does not require the technical recourse 
assumption referred to as ``(17)'' in~\cite{BCDELL17}.  This
difference results not from replacing the PH main loop with 
Algorithm~\ref{alg:aph}, but from using a nonlinear-first variant of FCFW
instead of the linear-first approach in~\cite{BCDELL17}.

\begin{proposition} \label{prop:phlikeconverge2}
Suppose Algorithm~\ref{alg:aph} is applied with $\tau=0$ to a problem in the
form~\eqref{discreteSPconvexified} meeting the assumptions of
Proposition~\ref{prop:phlikeconverge}, with Algorithm~\ref{alg:nffcfw} instead
of~\eqref{startstandardsubproblem}-\eqref{endstandardsubproblem} as the
subproblem solution subroutine $\mathcal{S}$.  Then the combined new algorithm
will have the same asymptotic convergence behavior asserted in{}
Proposition~\ref{prop:phlikeconverge}.
\end{proposition}
\begin{proof}
For every $i\in 1..\ell$, an element of $K_i$ is added to $R_i$ at each
iteration that the FW gap test $\biginner{\nabla h_{ik}(\bar x_i)}{\bar x_i -
\hat x_i} \leq \tau = 0$ fails to hold.  Since the sets $K_i$ are finite, at
most a finite number of elements may be added to each $R_i$, so there must exist 
some $\bar k$ such that for all iterations $k \geq \bar k$, all the sets
$R_i$ remain unchanged.  For all $i\in 1..\ell$ at such iterations, the point
$\hat x_i$ returned by the first ($l=0$) invocation of the MILP solver on
line~\ref{line:LMOcall} of Algorithm~\ref{alg:nffcfw} must return a point
already in $R_i$.  By Lemma~\ref{lem:incMeansGap}, this situation must result
in a zero Frank-Wolfe gap, and so the subroutine will immediately return with results based
on the current value of $\bar x_i$, which is an exact solution to the
subproblem by Lemma~\ref{lem:terminate}.

As a consequence, the combined algorithm will behave exactly like
Algorithm~\ref{alg:aph} with the exact subproblem solver
from~\eqref{startstandardsubproblem}-\eqref{endstandardsubproblem} at every
iteration $k \geq \bar k$.  Therefore, after iteration $\bar k$,
Proposition~\ref{prop:phlikeconverge} applies after renumbering its
iterations $k \leftarrow k - \bar k$, and that proposition's asymptotic
convergence results hold.
\end{proof}

\section{Parallel implementation}
\label{sec:implementation_approach}
This section describes our approach to implementing the combination of
Algorithm~\ref{alg:aph} and Algorithm~\ref{alg:nffcfw} in an efficient
asynchronous manner scalable to current High-Performance Computing (HPC) or
cluster platforms.  Such systems consist of HPC ``nodes'' connected by a
high-speed network, with each HPC node containing 16 to 128 CPU cores residing
on between 1 and 8 CPU chips.  All the
cores on an HPC node have direct access to the HPC node's entire shared
memory, but employing data stored on other HPC nodes requires use of the
network.\footnote{Here, we use the term ``HPC node'' instead of just ``node'' to
distinguish these hardware nodes from the nodes of the stochastic programming
scenario tree.}

To take advantage of such hardware, we designed our implementation to use
MPI~\cite{Gropp:1994:UMP,mpiforum2023mpi41}.  The fundamental computing cell
within MPI programs is called a ``rank,'' which corresponds to an operating
system process consisting of a single memory address space and one or more parallel
threads of control.  Ranks communicate with one another through MPI's various
messaging capabilities, which include both simple point-to-point messages and
efficient implementations of various collective operations such as global
summations of scalars or vectors.  Because all threads within a rank must have
access to the same shared memory, each rank must be confined within a single
HPC node.  The simplest arrangement is to have one rank per HPC node, each running
as many threads as the CPU cores on the HPC node can support.  However,
in some cases one may wish to pack multiple ranks onto one HPC node; communication
with ranks on the same HPC node uses the same software MPI interface as
communication with ranks on
other HPC nodes, but data are transparently transmitted through on-node
inter-process shared memory buffers rather than through the network.  Up to
modest rank counts, MPI implementations typically run efficiently on ordinary
multi-core CPUs, emulating a networked system.  In this case, all ranks are in
fact resident in the memory and processors of a single motherboard.

The asynchronous nature of Algorithm~\ref{alg:aph} required a fresh
implementation, albeit one that is
integrated into the existing \texttt{mpi-sppy} system~\cite{mpisppy23} already
used to implement various forms of PH, including a Frank-Wolfe PH algorithm
based on~\cite{BCDELL17}.  To control parallelism within each MPI rank, we use C++
native multithreading with the \texttt{<thread>} and \texttt{<future>}
libraries; see for example~\cite{williams2019cpp}.  We partition the scenarios
of the stochastic programming problem approximately equally among the MPI
ranks, much as in~\cite{mpisppy23}: each scenario is assigned to a rank
at initialization and remains there, being processed only by that rank.

\subsection{Application thread architecture}

When scenario data is spread across the ranks of a distributed system, the
projection operations in lines~\ref{line:uproj}-\ref{line:phlikewupdate} of
Algorithm~\ref{alg:aph} require coordination of all ranks, because they use
and change information associated with every scenario.  This portion of the
algorithm can only be implemented in an essentially synchronous way.  However,
each individual rank does not need to fully synchronize these coordination operations with
its work solving subproblems.  De-synchronizing coordination and subproblem
responsibilities within each rank is the key to our design of an essentially
bottleneck-free implementation, despite the projection component of
Algorithm~\ref{alg:aph} having a synchronous flavor.

Conceptually, our implementation maintains a single ``control thread'' and a
configurable number $\Theta_{\mathrm{w}} \geq 1$ of ``worker threads'' on each
rank.  The worker threads' only responsibility is to execute
Algorithm~\ref{alg:nffcfw} to solve subproblems, with each worker thread being
able to process any scenario on its rank.  Further, when invoking the MILP
solver on line~\ref{line:LMOcall}, each worker thread authorizes the MILP
solver to use a total of $\Theta_{\mathrm{m}} \geq 1$ threads (including itself) 
to concurrently
explore the branch-and-cut tree.  The highest number of threads in use on a
single rank will occur when all $\Theta_{\mathrm{w}}$ worker threads are
active and currently solving a MILP, with each MILP solve currently using its
full allocation of $\Theta_{\mathrm{m}}$ threads: the resulting active thread
count is $\Theta_{\mathrm{w}}\Theta_{\mathrm{m}} + 1$, the $1$ being due
to the control thread.  
If $\kappa$ ranks are resident on each HPC node, then $\kappa(
\Theta_{\mathrm{w}}\Theta_{\mathrm{m}} + 1)$ should typically approximately equal
the number of concurrent compute threads $\kappa^*$ that the HPC node can
efficiently support.  A number of factors may influence $\kappa^*$ for a
particular application and problem instance; sometimes $\kappa^*$ may be less
than the number of CPU cores on the HPC node due to memory bandwidth
considerations.

In our implementation, communication between ranks is performed only by the
control threads, using non-blocking collective operations.  MPI collective
operations are those in which every member of a designated set of ranks,
called a \emph{communicator}, must participate.  The only such operations our
implementation currently uses are ``all reductions'' with the sum operator,
meaning that each rank in the communicator contributes a vector of integer or
floating point numbers, and at the end of the operation, each rank in the
communicator receives the sum of the vector over all ranks in the
communicator.  In the main flow of our computation, we exclusively employ
non-blocking collective operations, meaning that each rank initiates the
communication and then continues to other tasks while periodically checking
for completion.  The operation only begins when every other rank in the
communicator starts participating, and then completes at some later time.

\begin{figure}[t]
\includegraphics[width=\textwidth]{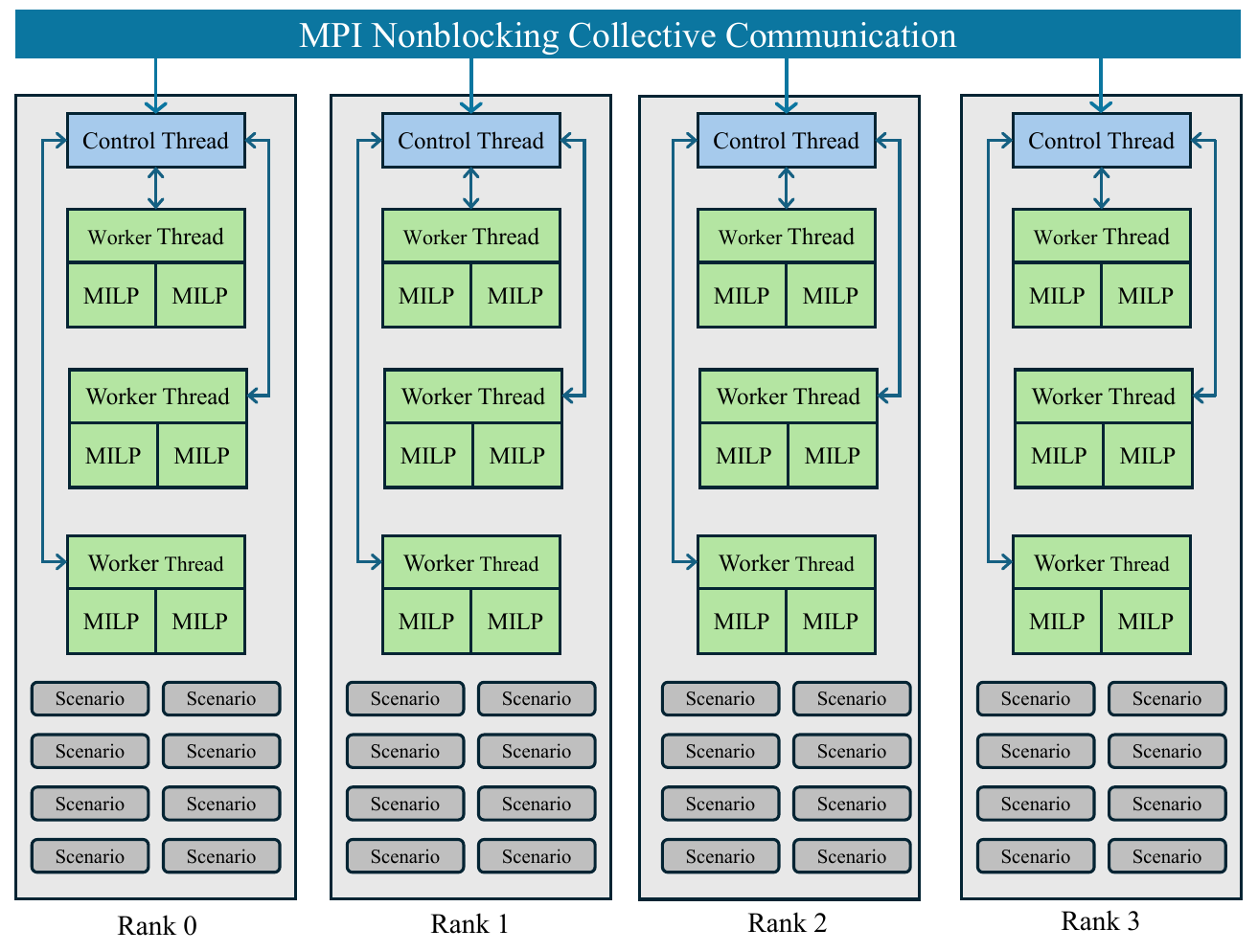}
\vspace{-4ex}
\caption{Schematic illustration of thread, rank, and communication organization.}
\label{fig:thread-org}
\end{figure}

Figure~\ref{fig:thread-org} depicts the thread and rank organization of the
implementation.  In the example shown in the figure, $8$ scenarios are
assigned to each rank, there are $\Theta_{\mathrm{w}} = 3$ worker threads per
rank, and each worker thread authorizes the MILP solver to use up to
$\Theta_{\mathrm{m}} = 3$ branch-and-cut threads.  Thus, at most $3 \cdot 3 +
1 = 10$ computation threads may be simultaneously active on any given rank
(however, additional threads will be used by the operating system on each node
and MPI background services on each rank).  As shown at the top of
the figure, all communication between ranks is by nonblocking collective
communication among the control threads.

\subsection{Scenario states}

Within each rank, each scenario cycles through four possible states, in the
order
\begin{equation*}
E \to A \to Q \to P \quad (\text{and then back to~} E),
\end{equation*}
interpreted as follows:
\begin{description}
\item[Eligible, $E$ ---] Eligible to be solved with the current multipliers $w^k$ 
and centers $z^k$ as input.
\item[Active, $A$ ---] Currently being solved by a worker thread.
\item[Queued, $Q$ ---] Solution complete, but not yet incorporated into any
projection / coordination calculations on
lines~\ref{line:uproj}-\ref{line:phlikewupdate} of Algorithm~\ref{alg:aph}.
In the notation of that algorithm, these scenarios will be members of $I_k$ during
the next projection / coordination cycle.
\item[In projection, $P$ ---]  Currently being used in ongoing projection /
coordination step, that is, part of $I_k$ for projection calculations
currently in progress.
\end{description}
At any given moment, every scenario assigned to a rank is in exactly one of
these four sets.

This design requires that the results of the most recent subproblem solve for
a given scenario $i$ must be incorporated into the iterates $w^k$ and $z^k$
before $i$ may be considered for solution again.  While this restriction is
not absolutely necessary to follow the template provided by
Algorithm~\ref{alg:aph}, it greatly simplifies the implementation,
preventing concurrent solves of the same scenario and thus eliminating
potential memory conflicts over the MILP and nonlinear solver data structures
associated with each scenario.

\subsection{Control event loop and coroutine structure}
\label{sec:eventLoop}

\begin{figure}[b]
\includegraphics[width=\textwidth]{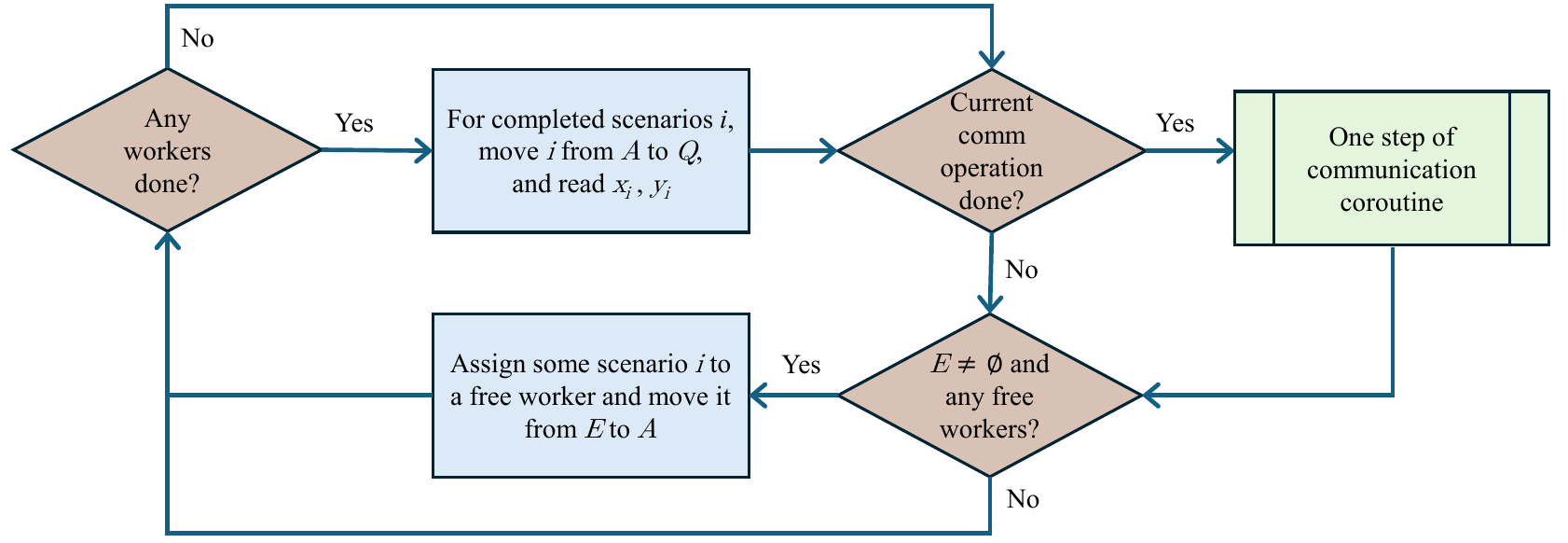}
\vspace{-4ex}
\caption{Flowchart of the principal operations in the control thread.}
\label{fig:control-loop}
\end{figure}

Figure~\ref{fig:control-loop} illustrates the main event loop executed by each
control thread.  The control thread repeatedly checks for the
following conditions and takes the following actions:
\begin{description}
\item[Completed scenario subproblem] 
(the ``yes'' branch from ``any workers done?'') 
If some subproblem $i$ has completed its
processing, incorporate its values of $x_i$ and $y_i$ into the projection
calculation data structures, move $i$ from $A$ to $Q$ (putting $i$ in the
queued state), and decrement the number of currently active subproblem solves
$\omega$.  The completion check triggering this operation polls a C++ thread status
object of type \texttt{std::future} associated with each scenario
in the active pool $A$.
\item[Completed communication operation]
(the ``yes'' branch from ``current comm operation done?'') If a communication
operation has just completed, call the communication coroutine (described
below) to advance to the next communication operation if necessary.  The check
triggering this operation uses \texttt{MPI\_Testall} to determine whether
all MPI operations associated with the currently running communication
operations have completed.
\item[Scenario dispatch]  If $E \neq \emptyset$ and $\omega <
\Theta_{\mathrm{w}}$, that is, there are scenarios eligible for solution and the
number of active subproblem solves is less than the maximum, remove some
scenario $i$ from $E$ using the safeguarded ``greedy'' heuristic
from~\cite[Section 5.3]{JE22}, initiate a subproblem solve on it using the
current values of $z_i^k$ and $w_i^k$, and increment the number of active
subproblem solves $\omega$.

To launch the asynchronous scenario solve, the
controller uses the \texttt{std::async} method from the C++ native
multithreading library, which returns a \texttt{std::future} status object
which may be subsequently polled to test subproblem completion.
\end{description}

The control thread loops indefinitely through these tests and actions until a
termination flag is raised.  If the last cycle of tests did not result in any
actions, the control thread sleeps for a configurable amount of time
defaulting to 20 milliseconds.  These pauses are intended to limit the CPU
load and memory traffic generated by the control thread when it is waiting for
actionable events to occur.  Further, if no communication is currently
ongoing, the control thread may elect to restart communication if a
configurable amount of time has passed or a subproblem has completed
processing.  For clarity and simplicity, the details of termination
detection, controller sleep, and communication restarts have been omitted from
the flowchart in Figure~\ref{fig:control-loop}.

Our actual control loop code was created by deriving from a generic
\texttt{groupController} class which we designed to manage any computation
pattern broadly similar to the APH method's.  This design should allow much of
our code to be reused in similar future applications, should any arise.

The communication coroutine mentioned above collectively cycles the control
threads of all the ranks through a network of communication states, as
depicted in Figure~\ref{fig:comm-states}. Aside from small amounts of
skew induced by communication latency, this cycling is essentially
synchronous.  These states are as follows:

\begin{figure}[t]
\includegraphics[width=\textwidth]{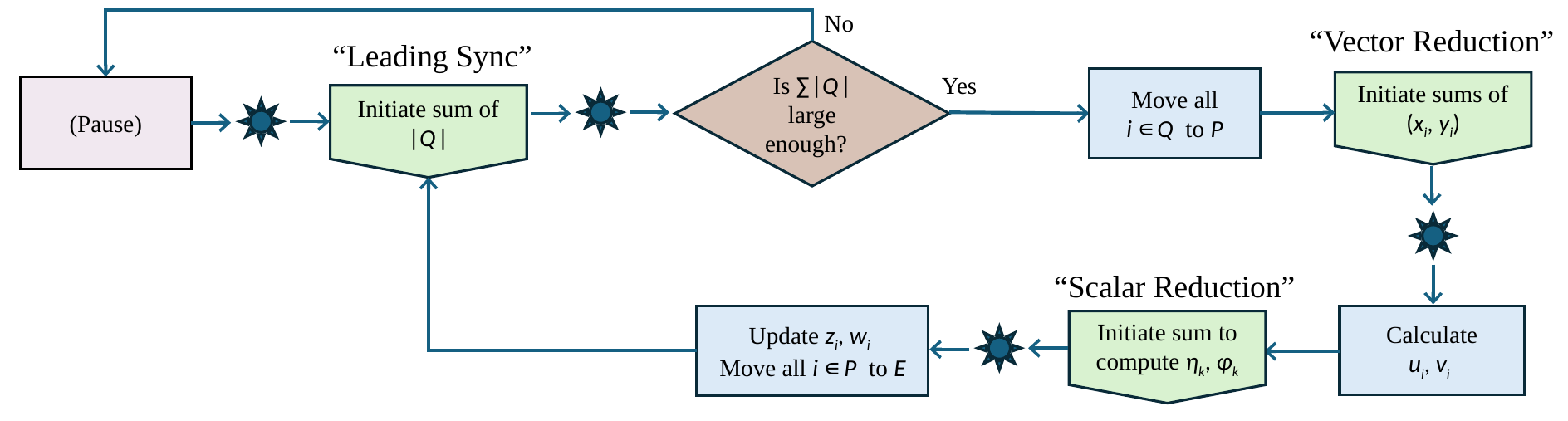}
\vspace{-5ex}
\caption{Flowchart of communication coroutine states and associated actions.}
\label{fig:comm-states}
\end{figure}

\begin{description}
\item[Leading sync]  In this communication state, the control threads
collectively check whether the number of completed subproblems $\card{Q}$, when
summed across all ranks, is sufficient to start a new coordination cycle.  If
not, communication pauses and then repeats the leading sync state.  If so, the
communication proceeds to the vector reduction state.
\item[Vector reduction] Calculates $u^k$ and $v^k$ on
lines~\ref{line:uproj}-\ref{line:vproj} of Algorithm~\ref{alg:aph}.
\item[Scalar reduction] Calculates $\eta_k$ and $\varphi_k$ on
lines~\ref{line:phlikedenom}-\ref{line:phlikenumer} of
Algorithm~\ref{alg:aph}.
\end{description}

The star symbols in Figure~\ref{fig:comm-states} indicate exit and reentry
points for the coroutine.  Every time the control thread main event loop in
Figure~\ref{fig:control-loop} invokes it, the coroutine starts
execution at one of the stars and exits at the next star it
encounters.  Before exiting, it initiates another communication operation (set to
trigger after a configurable delay in the case of the leftmost star in the
diagram).  Once this communication has completed, the next invocation of the coroutine resumes at the same star.  

We now describe further details of the processes described in the flowchart.
The MPI collective primitive \texttt{MPI\_Iallreduce} used in each step is a
non-blocking version of the all-reduce operation; this operation takes a
scalar or vector contributed by each rank in a communicator, combines them with
a specified operator (addition in all the cases used here), delivering the
result to all ranks in the communicator.  In most MPI environments, its time
complexity is logarithmic in the number of ranks in the communicator.

\paragraph{Implementation of leading sync.}  The leading sync uses
a summing \texttt{MPI\_Iallreduce} on a communicator consisting of all the
ranks, computing the sum of the integers $\card{Q}$, the number of completed
subproblems that have not yet been incorporated into the projection process.
Letting $\mathcal{R}$ be the set of all ranks and $Q_r$ be the set $Q$ on
rank $r$, the condition in the ``is $\sum\card{Q}$ big enough?'' decision
block in Figure~\ref{fig:comm-states} is more precisely written
\begin{equation} \label{solvefracond}
\sum_{r\in\mathcal{R}} \card{Q_r} \geq \sigma \ell,
\end{equation}
where $\sigma \in\; ]0,1]$ is a parameter called \texttt{solveFrac} in the
source code.  This parameter controls how aggressively the algorithm attempts to
update its principal working iterates $(z^k,w^k)$: if $\sigma$ is set very
small, for example $\sigma \leq 1/\ell$, then the method will attempt to
update the working iterates as soon as any new information becomes available
that could alter the separator used in the projection step.  Values of
$\sigma$ near $1$ will tend to limit how asynchronously the implementation
operates, since a projection will only be initiated once most of the scenarios
in the tree are in the queued state, and it is possible that some ranks will
run out of eligible scenarios to process while waiting
for~\eqref{solvefracond} to hold. The most extreme form of this phenomenon is
when $\sigma = 1$, in which case the implementation behaves synchronously,
waiting for every subproblem to be solved before initiating each coordination
step.  In this mode, the algorithm behaves like a straightforward parallel implementation
of the PH-FW algorithm of Boland \emph{et al.}~\cite{BCDELL17}, but with a
more complicated coordination procedure.  With intermediate values of
$\sigma$, the method will not attempt to use the communication network as
intensively as when $\sigma \leq 1/\ell$, but worker threads should only
rarely or briefly run out of eligible subproblems to work on, so the algorithm
will run essentially without bottlenecks.

\paragraph{Implementation of vector reduction.}  
At the beginning of the vector reduction step, all queued subproblems are
changed to the ``in projection'' state, that is, $P \leftarrow Q$ and $Q
\leftarrow \emptyset$. Thereafter, the vector reduction step's responsibility
is to compute $u^k$ and $v^k$ at lines~\ref{line:uproj}-\ref{line:vproj} of
Algorithm~\ref{alg:aph}.  Doing so in turn involves the
computations~\eqref{phprojformula} for $v^k$ and similar calculations embedded
within~\eqref{conjproj} for $u^k$.  Using a slightly more compact, vector
notation that omits the individual coefficient indices $j$, these calculations
are
\begin{align} 
(\forall\,N\in \mathcal{U}) \;
\big(\forall\,i\in \mathcal{D}(N)\big) && \quad
u_{it(N)}^{k} &= x_{it(N)}^{k} - \frac{1}{\pi(N)}
   \sum_{i'\in \mathcal{D}(N)} \!\!\! \pi_{i'} x_{i't(N)}^{k} 
&&&&&& \label{vecprojformulau} \\
&&
v_{it(N)}^{k} &= \frac{1}{\pi(N)}
   \sum_{i'\in \mathcal{D}(N)} \!\!\! \pi_{i'} y_{i't(N)}^{k}. \label{vecprojformulav}
\end{align}
Let $\overline{\mathcal{U}} \subseteq \mathcal{U}$ denote the non-leaf nodes
whose scenarios are assigned to more than one rank.  Our implementation
performs the portions of the
calculations~\eqref{vecprojformulau}-\eqref{vecprojformulav} that are not
local within ranks by launching $\card{\overline{\mathcal{U}}}$ concurrent
non-blocking MPI all-reductions, one for each $N \in \overline{\mathcal{U}}$.
The procedure broadly resembles the one used for the $\mathcal{N}$-projection
operation in \texttt{mpi-sppy}~\cite{mpisppy23}, except that the reduced vectors
are twice as long because two algorithm vectors $u^k, v^k \in
\real^m$ must be calculated instead of a single vector
$\proj_{\mathcal{N}}^{(\pi)}(x^k) \in \real^m$.  Let $\mathcal{I}(r)$ denote
the set of scenarios assigned to MPI rank $r$.

\paragraph{Implementation of scalar reduction.}  After detecting completion
of all the reductions required for the vector reduction phase, the next
communication step is scalar reduction, in which the principal work is
computing $\eta_k$ and $\varphi_k$ on 
lines~\ref{line:phlikedenom}-\ref{line:phlikenumer} of Algorithm~\ref{alg:aph}.  
The non-rank-local portions
of these calculations are implemented by a single reduction over all ranks
with vector length two.
The final stepsize determination on line~\ref{line:phlikestepsize}, 
requiring only a few scalar operations, is replicated across all ranks.
The actual projection steps on
lines~\ref{line:phlikezupdate}-\ref{line:phlikewupdate} may then be executed
in a parallel, distributed manner across all ranks, with no further
communication required.  After the projection, all scenarios in $P$ are moved to
$E$ to indicate they are once again eligible for processing, and $P$ becomes
empty. The communication state returns to the leading sync phase, to check
whether enough newly solved scenarios have accumulated to start another
projection.

\subsection{Conformance to the assumptions of Propositions~\ref{prop:phlikeconverge}
            and~\ref{prop:phlikeconverge2}} 

This section discusses the conditions under which the convergence result in
Proposition~\ref{prop:phlikeconverge2} applies to the implementation; these
conditions are identical to those of Proposition~\ref{prop:phlikeconverge}
other than the choice of subproblem subroutine $\mathcal{S}$.

For this discussion, it is first necessary to describe in detail the
``safeguarded greedy'' subproblem selection heuristic from~\cite[Section
5.3]{JE22} that the control thread uses to choose the next scenario to be
solved. In the current setting and notation, the safeguarded greedy selection
method works as follows: for some integer parameter $B' \geq 0$, form for each
$i\in E$ the two-element vector
\begin{equation} \label{selection}
\begin{bmatrix}
\min\big\{0, B' - k + d(i,k) \big\} \\
\pi_i (z_i^k-x^k_i)\transpose(w^k_i-y_i^k)
\end{bmatrix},
\end{equation}
where $d(i,k)$ is the last iteration at which scenario $i$ was dispatched for
solution; then, select the $i$ for which the vector is lexically smallest.  If
every eligible subproblem has been dispatched within the last $B'$ iterations,
the first element of all these vectors will be zero, and priority will be
given to the smallest value of $\pi_i (z_i^k-x^k_i)\transpose(w^k_i-y_i^k)$; a
rationale for this priority is given in~\cite[Section 5.3]{JE22}.  However, if
there exists an eligible scenario that has not been dispatched for processing
within the last $B'$ iterations, then one of the ``oldest'' subproblems, in the
sense of having the smallest value of $d(i,k)$, will be chosen.

From the form of this procedure, it is clear that any given subproblem $i$
will be selected at least once in every $B' + 1 + \card{\mathcal{I}(r)} \leq
B' + 1 + I_{\max}$ times that rank $r$ has a subproblem to dispatch, where
$I_{\max} \doteq \max\big\{\card{\mathcal{I}(r)} : r\in \mathcal{R}\big\}$ is
the largest number of scenarios assigned to a rank.

Further, assume that there is some minimum time $t_{\min}$ and maximum time
$t_{\max}$ that it takes to process a scenario.  Then the implementation falls
under the class of generic asynchronous projective splitting implementations
envisioned in~\cite{Eck17} and in particular conforms to the assumptions of
Propositions~\ref{prop:phlikeconverge} and~\ref{prop:phlikeconverge2} by the
arguments given in~\cite[Proposition 3.1]{Eck17} (with some changes of
notation).

\subsection{Startup and shutdown}
\label{sec:startupShutdown}

Our implementation
starts with an ``Iter0'' procedure that synchronously solves every subproblem
without any objective function perturbations. Specifically, each rank $r$
solves $\min_{x_i\in K_i} h_i(x_i)$ for each $i\in\mathcal{I}(r)$. The initial
iterate $z^0$ is then set to the probability-weighted projection of the
resulting ensemble of solutions $\{x_i^0\}_{i=1}^\ell$ onto the
nonclairvoyant subspace $\mathcal{N}$, using one MPI reduction per
non-leaf node.
This procedure has some practical advantages: it verifies that every 
subproblem is solvable and gives every scenario $i$ a reasonable starting
value for $z_i$. Further, the probability-weighted sum of Iter0 objective
values yields a Lagrangian lower bound on the objective with $w = 0$,
equivalent to the optimal value of the relaxation of~\eqref{discreteSP} with
the constraints $Mx \in \mathcal{N}$ removed.  This initial bound may be computed with
minimal incremental effort, essentially a single scalar reduction operation.
After the Iter0 procedure is complete, the control threads enter the main
event loop described in Section~\ref{sec:eventLoop} and the algorithm begins
operating asynchronously.

Termination is detected at the
end of the scalar reduction phase of the communication coroutine.  One
condition for termination is that the scalar $\eta_k$ computed on
line~\ref{line:phlikedenom} of Algorithm~\ref{alg:aph} is below some small
tolerance, which means that the primal vector $x^k$ is very close to being
nonclairvoyant and $y^k$ is nearly in $\mathcal{N}^*$, implying that it is very
close to being a valid choice for the Lagrange multipliers.  
At the same point in the code that tests $\eta_k$ for termination, the current
implementation also checks whether a specified maximum iteration
count or wall clock time has been reached. 
When our method is running as a hub solver within \texttt{mpi-sppy},
termination can also be triggered by conditions checked elsewhere in the current 
collection of \texttt{mpi-sppy} processes; however, the experiments in the next section
only consider the behavior of our algorithm in isolation.

If termination is detected, the scalar reduction logic in the communication
coroutine raises a termination flag.  All our termination checks are based on
quantities that are the same in every rank at the end of the scalar reduction
step, so that the decision to terminate is necessarily unanimous among all ranks.

Once the termination flag is raised, the control thread signals an abort to
any subproblem solves currently in progress: the underlying solver is interrupted
asynchronously using its thread-safe termination mechanism, the Frank-Wolfe
inner loop stops at the next opportunity, and any partial results are
discarded.  Letting $w^{\mathrm{last}}$ denote the value of $w^k$ when termination is detected, each rank then embarks on computing
\begin{equation*}
\lambda_i \doteq 
  \min_{x_i\in K_i} \big\{ h_i(x_i) + (w_i^{\mathrm{last}})\transpose \overline{M} x_i \big\} 
\end{equation*}
for each scenario $i\in \mathcal{I}(r)$ assigned to it.  These solves are
initiated as soon as any free worker threads are available, even if some
main-event-loop scenario solves have not yet terminated.  Once each rank $r$ has
computed all its $\lambda_i, i\in \mathcal{I}(r)$, it computes the sum
$\sum_{i\in \mathcal{I}(r)} \pi_i \lambda_i$.  It then enters this value into
a single (blocking) sum reduction that computes the Lagrangian
bound~\eqref{lagdualw} for $w = w^{\mathrm{last}}$, serving as the run's
final estimate of the best Lagrangian bound.  When our implementation runs
as an \texttt{mpi-sppy} hub process, the final Lagrangian computation 
may be skipped, since an \texttt{mpi-sppy} Lagrangian spoke process may be used 
to compute bounds continuously as the method runs.

In summary, the startup operation is synchronous.  
When computing the final Lagrangian bound, the
termination procedure becomes synchronous at the final blocking reduction;
when running as an \texttt{mpi-sppy} hub, the final Lagrangian computation 
may be skipped entirely, so that
shutdown involves no synchronization beyond end-of-run
performance statistics reductions.  Between startup and shutdown,
the implementation runs asynchronously to the extent allowed by~$\sigma$ 
and a runtime parameter \texttt{-{}-block-sync} described below.

\section{Sample computational testing}
\label{sec:computational_testing}
This section reports some small-scale computational testing on an ``SSLP'' test
problem~\cite{NS05,NS08,SIPLIB}.  At the time of writing, plans for much 
larger-scale testing are being formulated.  
The SSLP instances are two-stage problems designated
\texttt{SSLP\_}$s$\texttt{\_}$c$\texttt{\_}$\ell$.  Such an instance has $s$
first-stage variables, all binary, indicating whether to install servers at
each of $s$ possible locations.  There are then $\ell$ possible demand
scenarios distributed over $c$ clients to be served from the facilities, with
the second stage having a mix of binary and continuous variables, the majority
being binary.  The particular problem instance presented here is 
\texttt{pysp\_sslp\_10\_50\_50}, which has $10$ candidate servers,
$50$ clients, and $50$ equally likely scenarios.  Its data differ slightly
from the \texttt{sslp\_10\_50\_50} instance in
SIPLIB,\footnote{\texttt{Pysp\_sslp\_10\_50\_50} is the version of the
instance distributed with PySP~\cite{watson2012pysp} and read by
\texttt{mpi-sppy}'s \texttt{sslp} example.  Its data differ from SIPLIB's
only for client $1$ at servers $6$ through $10$, where revenue and demand
coincide (as they do throughout both versions of the instance): 
these values are $3, 22, 20,
1$, and $10$ in the PySP data (a copy of client $2$'s values at servers $1$
through $5$), versus $14, 5, 2, 9$, and $1$ in SIPLIB. } 
but we observed very similar computational results for the SIPLIB
version of the instance.  
The $10$ server-location binaries are the only
nonclairvoyant variables; the remaining variables are $500$ client-assignment
binaries, one per server-client pair, and $10$ continuous overflow variables.
Each scenario subproblem thus has $60$ constraints and $520$ variables, $510$
of them binary.

On this instance, we tested $6$ algorithm configurations.  Each test was run on a
$16$-thread, $8$-core laptop, dedicating $14$ threads to solving MILP problems
with the standard Gurobi 13.0.2 solver.  The algorithm configurations were:

\paragraph{The PH benchmark.}  The ``PH benchmark'' is the Frank-Wolfe progressive
hedging (FWPH) hub of \texttt{mpi-sppy}~\cite{mpisppy23}, an
implementation based on the algorithm of Boland \emph{et
al.}~\cite{BCDELL17}, using progressive hedging in the outer loop,
with each subproblem solved by a fully corrective Frank-Wolfe method.
Like most of \texttt{mpi-sppy}, it is implemented in Python with heavy
use of libraries and solvers written in compiled languages.  We ran
this configuration with $14$ MPI ranks, each using a single thread to
solve one MILP at a time, otherwise with \texttt{mpi-sppy}'s default settings.

\paragraph{Various configurations of APH-FW.}  The remaining configurations
used the C++ implementation of APH-FW described in
Section~\ref{sec:implementation_approach}, with one Frank-Wolfe iteration per
subproblem solve ($l_{\max} = 1$) in Algorithm~\ref{alg:nffcfw}. These runs
used $7$ MPI ranks, each with $2$ workers, under $5$ different parameter
configurations:
\begin{itemize}
\item $\sigma = 1$, which we call the ``APH baseline'', which waits for
completion of every scenario before each projection and is therefore synchronous;
\item $\sigma = 0.7$, $0.5$ and $0.3$, which project once that share of the
  scenarios has been solved, displaying both block and lag asynchrony;
\item $\sigma = 0.3$ ``forced synchronous,'' which modifies the $\sigma = 0.3$
  configuration with the runtime parameter \texttt{-{}-block-sync}.  In this
  mode, the control thread for each rank $r$ stops dispatching new subproblems
  after the first $\lceil \sigma \card{\mathcal{I}(r)} \, \rceil$ solves in
  each projection cycle, waiting until all of them and the next projection
  step have completed to dispatch more.  This configuration results in block
  asynchrony, but no lag asynchrony.
\end{itemize}

Because the problem instance has only $2^{10} = 1024$ choices of the nonclairvoyant
variables, solving every scenario at each of them can be used to certify both
the optimal value, $-369.94$, and the Lagrangian dual value, which in this
case is also $-369.94$ (so the instance has a zero Lagrangian duality gap).

Every configuration uses $\rho = 1$ for every nonclairvoyant variable and runs
$14$ single-threaded subproblem solves with a time limit of $360$ seconds: the PH
benchmark on $14$ ranks, APH-FW on $7$ ranks of $2$ workers, with
the primal-dual scaling parameter $\xi$ at its default of $1$.  
All configurations solve their Frank-Wolfe linear minimizations at a relative
MIP gap of $10^{-4}$ and use the \texttt{mpi-sppy} objective cutoff option
(a constraint bounding each linear minimization's objective by its value at
the current point, plus a relative slack of $10^{-4}$).

Results are shown in Figure~\ref{fig:ScaledPysp}.  Its horizontal axis is
wall-clock time in seconds, and the vertical axis is the value of the best
Lagrangian bound found so far at the current time.  APH-FW only computes
Lagrangian bounds at the end of its run since they are not a natural byproduct
of its basic calculation cycle, so the displayed values were computed from a
log file after termination.  ``Solve fraction'' in the legend denotes the
value of $\sigma$.

\begin{figure}[tbp]
\vspace{1ex}
\centering
\includegraphics[width=\textwidth]{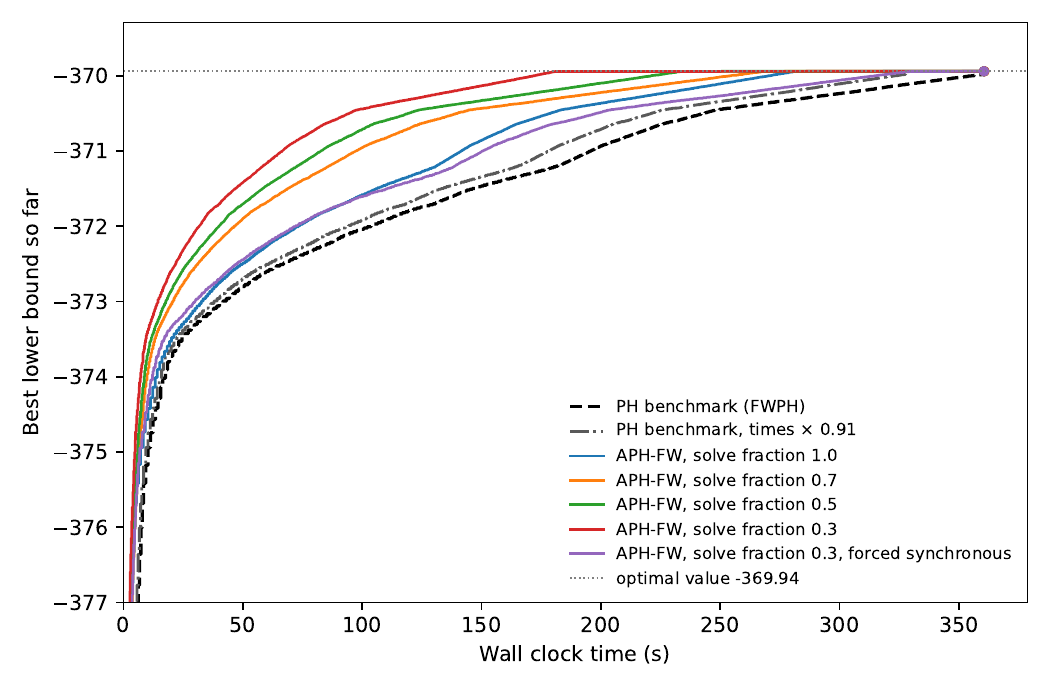}
\vspace{-6ex}
\caption{Lagrangian bound value plotted against wall clock time for
        \texttt{pysp\_sslp\_10\_50\_50}. }
\label{fig:ScaledPysp}
\end{figure}

On this instance, we measured about $9\%$ of the PH benchmark's run time to be
spent in Python,
so a C++ implementation would be at most about
$9\%$ faster.  The figure therefore shows the PH benchmark twice: as measured
(the black dashed line) and with its times multiplied by $0.91$ (the gray
dotted-dashed line); the actual performance of a C++ implementation of the PH
benchmark directly comparable to our C++ implementation of APH-FW would be
expected to lie somewhere between these two curves.

The PH benchmark, with or without the Python-versus-C++ adjustment to its
runtime, is the slowest algorithm variant to reach the optimal value,
indicated by the horizontal dotted line.  
The fastest configuration in the tests is APH-FW with $\sigma=0.3$, followed
in order by $\sigma=0.5$, $\sigma=0.7$, and the $\sigma=1.0$ APH baseline.
Notably, $\sigma=0.3$ in the forced-synchronous mode is the slowest APH
configuration, initially producing slightly better bounds than the APH
baseline, but then falling behind. Time-averaged over each run, the number of
Frank-Wolfe MILPs being concurrently solved is $12.1$ for $\sigma = 0.3$,
$11.4$ for $\sigma = 0.5$, $10.9$ for $\sigma = 0.7$, $8.6$ for the $\sigma =
1.0$ APH baseline, and $7.1$ for $\sigma = 0.3$ forced synchronous.  These
statistics follow the same ordering as the apparent speed of convergence in
Figure~\ref{fig:ScaledPysp}.
The $\sigma = 0.3$ configuration was able to keep the MIP solver threads busy
approximately $12.1/14 \approx 86\%$ of the time, whereas the APH baseline
was only able to occupy them about $8.6/14 \approx 61\%$ of the time. Taken
together, the results suggest that there is significant value in
asynchronism.

\section{Conclusions and future work}
\label{sec:conclusions_and_future_work}
Our computational results suggest that, in contexts with highly variable
subproblem solution times such as when computing Lagrangian bounds on
stochastic integer programming problems, asynchronous operation of a
decomposition method can yield significant performance improvements.  In our
case, provably convergent asynchronism is enabled by the underlying algorithmic
results in~\cite{ComEck15}, meaning that the maximum asynchronism lag does not
constrain any of the algorithm parameters as it does in alternative methods
like the ARock-based algorithm in~\cite{BLetal20}.  Our method's lag independence 
is enabled by the synchronous nature of the updates on
lines~\ref{line:uproj}-\ref{line:phlikewupdate} of Algorithm~\ref{alg:aph}.
However, this apparently synchronous feature does not prevent a careful
implementation from operating close to fully asynchronously, as
demonstrated by our implementation using within-rank multithreading and
nonblocking MPI collective operations.

Involving just one problem instance on a laptop, the computational results
presented in Section~\ref{sec:computational_testing} are highly preliminary,
and work to expand and improve them is ongoing.  In particular, we plan to
conduct much larger-scale runs on much larger stochastic integer programming
instances, especially instances for which there are no practical
alternatives to using some kind of decomposition method.  We also plan on runs
using many more processor cores and on multiple families of test problems.

We also envision eventual improvements to our implementation, such as
computing progressively improving upper and lower bounds on the optimal
objective value throughout a run, as opposed to a single lower bound at the
end as at present. This capability could be added through the existing
``cylinder'' features of \texttt{mpi-sppy}, but we are also contemplating
alternatives.  Further, if we find in some cases that the vector reduction
phase of our communication cycle limits the efficiency of our method, we may
also consider alternative ways of organizing that communication operation.

Finally, some theoretical improvements could also benefit our method. The
theory of approximate subproblem solutions in~\cite{ComEck15} and~\cite{Eck17}
uses an ``in-graph'' relative error criterion to prove convergence.  For our
application, however, as noted in~\cite{Eck2026gaps} in the case of PH and
related algorithms, an ``out-of-graph'' absolute error criterion would yield a
more realistic and satisfying convergence theory.  In this regard, recent
results such as those in~\cite{CM26} may be a helpful starting point.

\bibliographystyle{abbrv}
\bibliography{prox,general}

\end{document}